\documentclass{amsart}
\usepackage{amsmath,amsthm,amssymb,amsfonts}
\input xy
\xyoption{all}
\usepackage{hyperref}

\theoremstyle{plain}
\newtheorem{thm}{Theorem}[section]
\newtheorem{prp}[thm]{Proposition}
\newtheorem{lem}[thm]{Lemma}
\newtheorem{cor}[thm]{Corollary}

\newtheorem*{thmT}{Theorem T (Tate's theorem for fusion systems)}
\newtheorem*{thmY}{Theorem Y (Yoshida's theorem for fusion systems)}

\theoremstyle{definition}
\newtheorem{dfn}[thm]{Definition}

\theoremstyle{remark}

\newtheorem{exm}[thm]{Example}

\newcommand{\Hom}{\mathrm{Hom}}

\newcommand{\End}{\mathrm{End}}
\newcommand{\Aut}{\mathrm{Aut}}

\newcommand{\id}{\mathrm{id}}
\newcommand{\incl}{\mathrm{incl}}

\newcommand{\Img}{\mathrm{Im}}

\newcommand{\Ker}{\mathrm{Ker }}

\newcommand{\CF}{\mathcal{F}}

\newcommand{\CH}{\mathcal{H}}

\newcommand{\CN}{\mathcal{N}}

\newcommand{\la}{\langle}
\newcommand{\ra}{\rangle}

\newcommand{\ol}{\overline}

\def\FF{\ensuremath{\mathbb{F}}}

\newcommand{\pcomp}[1]{{#1}^\wedge_p}

\newcommand{\res}{\mathrm{res}}
\newcommand{\tr}{\mathrm{tr}}

\newcommand{\N}{\mathrm{N}}
\newcommand{\Z}{\mathrm{Z}}

\newcommand{\normal}{invariant }
\newcommand{\ZZ}{\mathbb{Z}}

\newcommand{\Ff}{\mathcal{F}}
\newcommand{\Hh}{\mathcal{H}}
\newcommand{\Nn}{\mathcal{N}}

\begin{document}
\title[Tate's and Yoshida's theorem for fusion systems]{Tate's and Yoshida's theorem on control of transfer for fusion systems}
\author{Antonio D\'iaz}
\author{Adam Glesser}
%\author{Nadia Mazza}
\author{Sejong Park}
\author{Radu Stancu}
\date{\today}
\begin{abstract}
We prove analogues of results of Tate and Yoshida on control of transfer for fusion systems. This requires the notions of $p$-group residuals and transfer maps in cohomology for fusion systems. As a corollary we obtain a $p$-nilpotency criterion due to Tate.
\end{abstract}
\maketitle

\section{Introduction}
In the theory of finite groups, the focal subgroup of a Sylow $p$-subgroup is determined entirely by $p$-fusion and detects whether the 
whole group $G$ has a nontrivial $p$-group quotient.  Moreover, under  
certain conditions, some subgroups of $G$ containing its Sylow $p$-subgroup determine the focal subgroup and hence whether $G$ has a  
nontrivial $p$-group quotient.  This phenomenon is traditionally  
called control of transfer; indeed these results can be obtained by  
using transfer maps in group cohomology.

A fusion system is a category $\CF$ whose objects are the subgroups of  
a fixed finite $p$-group $S$ and whose morphisms behave like  
conjugation maps in finite groups having $S$ as a Sylow $p$-subgroup.   
First introduced by Puig~\cite{puig:frobeniussystems},\cite{puig:frobeniuscategories} and further developed by Broto, Levi and  
Oliver~\cite{BLO2003theory}, fusion systems constitute a useful framework for studying the local theory of (blocks  
of) finite groups and $p$-local homotopy theory.  Hence it is a  
natural question whether and how classical results of local group  
theory can be extended to fusion systems.

Given a fusion system, one defines the focal subgroup (and other  
related subgroups like the hyperfocal subgroup) analogously to the group case. Moreover,  
these related constructs display the same  
key properties as their group theoretic counterparts (\cite{BCGLO2007extensions}, see  
also appendix.) In particular, using the characteristic elements of a fusion system, introduced in \cite{BLO2003theory} and refined in \cite{Ragnarsson2006spectra}, we define an appropriate notion of transfer maps in the cohomology of fusion systems.

Using these tools, we generalize to fusion systems two classical  
theorems on control of transfer in finite groups, one due to Tate and the other due to Yoshida. Tate's theorem, reformulated as in \cite{GagolaIssacs2008tate}, concerns three types of residuals of a finite group $G$: the elementary abelian $p$-group residual, the abelian $p$-group residual and the $p$-group residual. It states that, for a subgroup $H$ of $G$ containing a Sylow $p$-subgroup of $G$, $H$ has isomorphic residual to that of $G$ of one of these types if and only if $H$ does so for the three types. In any of these three cases, then, we say that $H$ \textit{controls transfer} in $G$. Yoshida's theorem \cite[Theorem 4.2]{Yoshida1978} says that if $S$ is a Sylow $p$-subgroup of $G$, then $\N_G(S)$ controls transfer in $G$ unless the wreath product $C_p \wr C_p$ is a quotient of $S$.

To generalize these results to fusion systems, we first need appropriate notions of residuals. The case of the $p$-group residual is handled in \cite{BCGLO2007extensions}, where the authors define the notion of a fusion subsystem of $p$-power index. %Furthermore,  in \cite[Theorem 4.3]{BCGLO2007extensions} %it is shown that for a saturated fusion system $\CF$ over the $p$-group $S$, the $p$-power index fusion subsystems of $\CF$ are in bijection with the %subgroups of $S$ containing the hyperfocal subgroup $O^p_{\CF}(S)$ of $\CF$. 
This motivates the following definition.

\begin{dfn}
Let $\CF$ be a saturated fusion system on a finite $p$-group $S$. 
\begin{enumerate}
\item $O^p_{\CF}(S)=\la [P,O^p(\Aut_\CF(P))]\mid P\leq S\ra$ (the \textit{hyperfocal subgroup} of $\CF$).
\item $A^p_{\CF}(S)=[S,\CF]=\la [P,\Aut_\CF(P)]\mid P\leq S\ra$ (the \textit{focal subgroup} of $\CF$).
\item $E^p_{\CF}(S)=\Phi(S)[S,\CF]=\Phi(S)O^p_{\CF}(S)$ (the \textit{elementary focal subgroup} of $\CF$). 
\end{enumerate}
\end{dfn}
%In Corollary~\ref{cor:focalhyperfocal}, we prove that $[S,\CF]=[S,S]O^p_{\CF}(S)$. The elementary abelian $p$-group residual, abelian $p$-group residual %and  $p$-group residual of $\CF$ are hence the quotient fusion systems of $\CF$ by the strongly $\CF$-closed subgroups  $E^p_{\CF}(S)$, %$A^p_{\CF}(S)$ and $O^p_{\CF}(S)$ respectively. As expected, these quotients fusion systems are the trivial fusion systems on the $p$-groups %$S/E^p_{\CF}(S)$, $S/A^p_{\CF}(S)$ and $S/O^p_{\CF}(S)$ respectively. 
Using Corollary \ref{cor:focalhyperfocal}, we have that $O^p_\CF(S) \subseteq A^p_\CF(S) \subseteq E^p_\CF(S)$ and that the former two groups are completely determined by $O^p_\CF(S)$ and $S$. Consequently, the interesting part of the following theorem, which is a generalization of Tate's theorem from \cite{GagolaIssacs2008tate}, is the implication $(\ref{intro_thm_Tate:elem}) \Longrightarrow (\ref{intro_thm_Tate:hyper})$.

\begin{thmT}\label{intro_thm_Tate} 
Let $\CF$ be a saturated fusion system on a finite $p$-group $S$, and let $\CH$ be a saturated subsystem of $\CF$ on $S$. The following are equivalent.
\begin{enumerate}
\item $E^p_\CF(S)=E^p_\CH(S)$.\label{intro_thm_Tate:elem}
\item $A^p_\CF(S)=A^p_\CH(S)$.\label{intro_thm_Tate:foc}
\item $O^p_\CF(S)=O^p_\CH(S)$.\label{intro_thm_Tate:hyper}
\end{enumerate}
\end{thmT}

To show the implication $(\ref{intro_thm_Tate:elem}) \Longrightarrow (\ref{intro_thm_Tate:hyper})$, instead of Tate's original cohomological proof, we follow the strategy of Gagola and Isaacs in \cite{GagolaIssacs2008tate}, using transfer maps. As a corollary, we obtain a fusion system version of the $p$-nilpotency criterion suggested by Atiyah \cite{Tate1964nilpotent} and proved independently with alternative methods in \cite{CantareroSchererViruel}.
%The transfer in the setting of fusion systems was already defined and successfully used in \cite{{DGMPtransfer}} by three of the authors and Nadia %Mazza. To prove Tate's theorem we needed to develop here further properties of the transfer.

\begin{cor}\label{intro_cor_p-nilpotency} %\marginpar{Where is this proved?}
Let $\CF$ be a saturated fusion system on a finite $p$-group $S$. If the restriction map $H^1(\CF;\FF_p)\to H^1(S;\FF_p)$ is an isomorphism, then $\CF=\CF_S(S)$.
\end{cor}

%In the theory of finite groups, transfer plays a key role in finding proper nontrivial normal subgroups. It is a weaker form of fusion: that is, if a %subgroup controls fusion, then it controls transfer. 
By analogy with the group case, if any of the equivalent statements in Tate's Theorem above hold, we say that $\CH$ \textit{controls transfer} in $\CF$. With this definition, the natural translation of Yoshida's theorem to fusion systems is, thus, given by the following theorem. 

\begin{thmY}\label{intro_thm_Yoshida}
Let $\CF$ be a fusion system on a finite $p$-group $S$ and let $\CH =  
N_\CF(S)$.  If $\CH$ does not control transfer in $\CF$, then $C_p \wr C_p$ is a
homomorphic image of $S$.
\end{thmY}

\textbf{Organization of the paper:} In Section \ref{S:biset}, we recall the notion of double Burnside rings and characteristic elements in order to define the transfer later in the same section. In Section \ref{S:Yoshida} we prove Yoshida's theorem for fusion systems. In section \ref{S:OpF}, we prove new properties of the the $p$-power index transfer that are needed in section \ref{S:Tate} to prove Tate's theorem for fusion systems and Corollary \ref{intro_cor_p-nilpotency}. In the appendix, we recall the definitions of \normal subsystems and quotient systems, and prove some of their properties in the $p$-power index case.

\textbf{Acknowledgments:} The authors are grateful to the Department of Mathematical Sciences at the University of Copenhagen for funding and hosting a workshop on fusion systems where the foundation for this work was established.  Antonio D\'iaz and Radu Stancu visited Mathematisches Forschungsinstitut Oberwolfach twice while Sejong Park was an Oberwolfach Leibniz Fellow, and worked on this project.  We thank the institute and the staff for their hospitality and excellent working environment.  We also sincerely thank K\'ari Ragnarsson for supplying us with a proof of a key property of the transfer in Theorem~\ref{RelationCharIdempPPower}.

\section{Characteristic elements and transfer for fusion systems}  
\label{S:biset}

\subsection{Double Burnside ring}
We begin this section with a brief review of the ($p$-localized) double Burnside ring of a finite group, following closely the treatment in \cite{Ragnarsson2006spectra}.  For finite groups $G$ and $H$, a {\em $(G,H)$-biset} is a finite set with commuting right $G$-action and left $H$-action.  The {\em Burnside module} $A(G,H)$ of $G$ and $H$ is the Grothendieck group of the monoid of isomorphism classes of $(G,H)$-bisets with free left $H$-action, under disjoint union. For finite groups $G$, $H$ and $K$ there is a bilinear map
\[
A(K,H)\times A(G,K)\to A(G,H)
\]
given by
\[
(\Omega,\Lambda)\mapsto \Omega\circ \Lambda:=\Omega\times_K \Lambda.
\]

As an abelian group, $A(G,H)$ is free with one generator for each isomorphism class of transitive $(G,H)$-bisets with free left $H$-action. These generators are represented by   bisets of the form $H \times_{(K,\psi)} G$, where $K\leq G$, $\psi \in \Hom(K,H)$ and
$$
H \times_{(K,\psi)} G = (H\times G)/\sim \text{, where $(x,uy)\sim(x\psi(u),y)$ for $x\in H$, $y\in G$, $u\in K$.}
$$
We use the notation $[K,\psi]^H_G$ to denote the generator corresponding to $H \times_{(K,\psi)} G$, and we write $[K,\psi]$ if $G$ and $H$ are clear from the context.  In case $G=H$, $A(G,G)$ becomes a ring, called the {\em double Burnside ring} of the group $G$. We will also consider its $p$-localization
$$
A(G,G)_{(p)}:=A(G,G)\otimes_{\ZZ}\ZZ_{(p)}.
$$
Note that $A(G,G)$ is a subring of $A(G,G)_{(p)}$.

For any $\ZZ G$-module $A$ there is a linear map 
$$
H^*(-;A)\colon A(G,G)\to \End(H^*(G;A))
$$ 
that takes the generator $[K,\psi]$ to 
$$
\tr^{G}_{K}\circ \psi^* : H^*(G;A)\to H^*(G;A),
$$ 
where $\tr^{G}_{K}:H^*(K;A)\rightarrow H^*(G;A)$ is the usual transfer map and $\psi^*:H^*(G;A)\rightarrow H^*(K;A)$ is restriction via $\psi$. It turns out that $H^*(-;A)$ is a ring homomorphism: for $\Omega, \Lambda \in A(G,G)$ we have
$$
H^*(\Omega\circ \Lambda;A)=H^*(\Omega;A)\circ H^*(\Lambda;A).
$$
If $A$ is a $\ZZ_{(p)}G$-module, the ring homomorphism
$$
H^*(-;A)\colon A(G,G)_{(p)}\to \End(H^*(G;A)).
$$
is defined analogously.

Now, let $\CF$ be a saturated fusion system over a finite $p$-group $S$. It is a remarkable result in the theory of fusion systems that there exist certain elements in $A(S,S)_{(p)}$, called \emph{characteristic elements}, that reflect all the properties of $\CF$ (see \cite{BLO2003theory} and \cite{Ragnarsson2006spectra}).  We discuss them below, and they are at the core of our definition of transfer for fusion systems.

We denote by $A_\CF(S,S)$ and $A_\CF(S,S)_{(p)}$ the subrings of $A(S,S)$ and $A(S,S)_{(p)}$, respectively, generated by $[P,\varphi]^S_S$ with $\varphi \in \Hom_\CF(P,S)$.  Let $\Omega\in A(S,S)_{(p)}$.  We say that {\em $\Omega$ is right $\CF$-stable} if for $P\leq S$ and every morphism $\varphi\in \Hom_\CF(P,S)$ the following equality holds in $A(P,S)_{(p)}$
$$
\Omega\circ [P,\varphi]^S_P=\Omega\circ [P,\incl]^S_P,
$$	
where $\incl:P\hookrightarrow S$ is the inclusion map. Left $\CF$-stability is defined analogously using the following equality in $A(S,P)_{(p)}$
$$
[\varphi(P),\varphi^{-1}]^P_S\circ \Omega=[P,\id]^P_S\circ \Omega,
$$	
where $\id:P\to P$ is the identity map. There is a unique linear extension $\epsilon$ to $A(S,S)_{(p)}$ of the map sending every generator $[P,\varphi]^S_S$ to its number of right $S$-orbits:
$$
\epsilon([P,\varphi]^S_S)=|S|/|P|.
$$
It is easy to see that, in fact, $\epsilon:A(S,S)_{(p)}\to \ZZ_{(p)}$ is a ring homomorphism and that it restricts to $\epsilon:A(S,S)\to \ZZ$.

\begin{dfn}\label{def_charbiset}
Let $\CF$ be a saturated fusion system over a finite $p$-group $S$. An element $\Omega\in A(S,S)_{(p)}$ is a \emph{characteristic element} for $\CF$ if  it satisfies the following properties:
\begin{enumerate}
\renewcommand{\theenumi}{\alph{enumi}}
\item $\Omega\in A_\CF(S,S)_{(p)};$\label{components}
\item $\Omega$ is right $\CF$-stable and left $\CF$-stable;
\label{invariant}
\item $\epsilon(\Omega) \not\equiv 0 \pmod {p\ZZ_{(p)}}$.  \label{coprime}
\end{enumerate} 
\end{dfn}

These three properties were first formulated by Linckelmann and Webb. In \cite[5.5]{BLO2003theory} Broto, Levi and Oliver proved that for any saturated fusion system $\CF$ there exists such a characteristic element $\Omega$, while in \cite{RagnarssonStancu2009idempotents}, Ragnarsson and Stancu prove that the existence of a characteristic element for a fusion system guarantees saturation. Furthermore, the element $\Omega$ constructed in \cite{BLO2003theory} is contained in $A_\CF(S,S)$ and has nonnegative coefficients; that is, it is an isomorphism class of an actual $(S,S)$-biset. We call such a characteristic element a \emph{characteristic biset} for $\CF$; more generally, if negative integral coefficients are allowed, we call it a \emph{virtual characteristic biset}. If $\CF$ is the fusion system induced by a finite group $G$ on its Sylow $p$-subgroup $S$ (i.e., $\CF=\CF_S(G)$) then $G$, viewed as an $(S,S)$-biset in the obvious way, is a characteristic biset for $\CF_S(G)
 $. See Example \ref{T:group transfer} for more details.

Characteristic elements of a given saturated fusion system $\CF$ are not unique.  Indeed, one can simply multiply a given characteristic element by a $p'$-number to get a new one.  But there is one special characteristic element introduced by Ragnarsson, which plays a key role in the theory.

%Note that given a characteristic element $\Omega$ and a positive integer $e$, one can obtain a characterstic element $\Omega'$ with $\epsilon(\Omega') \equiv 1 \mod p^e\Z_{(p)}$ simply by taking a disjoint union of a suitable $p'$-number of copies of $\Omega$.

%Note that, setting $I_0 =\{ i \in I \mid P_i = S \}$, we have
%\[
%       \epsilon(\Omega)=|\Omega| / |S| = \sum_{i \in I} |S:P_i| \equiv |I_0| \mod p.
%\]
%Thus \eqref{coprime} is equivalent to 
%\begin{enumerate}
%\renewcommand{\theenumi}{\alph{enumi}$'$}
%\setcounter{enumi}{2}
%\item $|I_0|$ is not divisible by $p$. \label{coprime'}
%\end{enumerate}
%Indeed, by the construction of the $(S,S)$-biset $\Omega$ in  \cite[5.4, 5.5]{BLO2003theory}, one can take $\Omega$ so that $\{ \varphi_i \mid i \in %I_0 \}$ is a set of coset representatives of $\Inn(S)$ in $\Aut_\CF(S)$.  Then \eqref{coprime'} follows from the Sylow axiom of fusion systems.

\begin{dfn}\label{def_charidem}
Let $\CF$ be a saturated fusion system over the $p$-group $S$. A \emph{characteristic idempotent} for $\CF$ is a characteristic element for $\CF$ that is an idempotent in the ring $A(S,S)_{(p)}$.
\end{dfn}

Note that the idempotent condition implies that %\marginpar{This paragraph needs work. For instance, the first sentence is isolated, M is not used in line 5, we should probably explain what $\pcomp{\ZZ}$ is and we should use either $\omega$ or $\omega_\CF$, not both.} 
$\epsilon(\omega)=1$.  In \cite{Ragnarsson2006spectra}, Ragnarsson shows that there exists a unique characteristic idempotent $\omega_\CF$ for every saturated fusion system $\CF$. We briefly recall here Ragnarsson's construction of $\omega_\CF$ (see (\cite[4.9, 5.8]{Ragnarsson2006spectra})) as it will be needed later. Given any virtual characteristic biset $\Omega\in A_\CF(S,S)$ for $\CF$, there is a large enough integer $M$ such that $\Omega^M$ is an idempotent modulo $p$. Then the sequence $\Omega^M$, $\Omega^{Mp}$, $\Omega^{Mp^2}$, $\ldots$ converges in the $p$-adic topology to an idempotent in $\pcomp{A(S,S)}:=A(S,S)\otimes_{\ZZ}\pcomp{\ZZ}$, where $\pcomp{\ZZ}$ are the $p$-adic integers. By uniqueness this idempotent has to be the characteristic idempotent $\omega_\CF$, and it turns out that $\omega_\CF$ actually lives in $A(S,S)_{(p)}$.

\subsection{Transfer}
We devote the rest of the section to defining the transfer map for fusion systems using characteristic elements and to proving some basic properties.  In particular,  we will show that the definition is essentially unique in spite of the choice of characteristic elements.  

Fix a saturated fusion system $\CF$ on a finite $p$-group $S$.  Let $A$ be a $\ZZ_{(p)}S$-module and consider a characteristic element $\Omega\in A_\CF(S,S)_{(p)}$ for $\CF$ expressed as
$$
\Omega=\sum  c_{[P,\varphi]} [P,\varphi],
$$
where the sum runs over the generators $[P,\varphi]$ of $A(S,S)$ and $c_{[P,\varphi]}\in \ZZ_{(p)}$. The endomorphism $H^*(\Omega;A)$ of $H^*(S;A)$ can be explicitly described as 
\begin{equation}\label{transfer}
H^*(\Omega;A)=\sum c_{[P,\varphi]}\cdot (\tr^{S}_{P}\circ \varphi^*).
\end{equation}

The following example highlights the feature of finite groups that $H^*(\Omega;A)$ is modeling.

\begin{exm} \label{T:group transfer}
Let $G$ be a finite group with Sylow $p$-subgroup $S$ and let $\CF =  
\CF_S(G)$.  The biset $\Omega = G$, where the $(S,S)$-biset structure is given by left and right multiplication in the group $G$, is a characteristic biset for $\CF$. An easy calculation shows that
\[
        \Omega \cong \coprod_{g \in [S\backslash G/S]} S \times_{(S \cap  
{}^{g}S,c_{g^{-1}})}S,
\]
and hence we get
\[
        H^*(\Omega;A) = \sum_{g \in [S\backslash G/S]} \tr^{S}_{S \cap {}^{g}S} \circ c^*_{g^{-1}}.
\]
But this is just the Mackey decomposition formula for the double cosets $SgS$ in $G$. Therefore,
\[
        H^*(\Omega;A) = \res^G_S \circ \tr^{G}_{S}
\]
where $\res^G_S:H^*(G;A)\rightarrow H^*(S;A)$ is restriction via the inclusion $S\hookrightarrow G$.
%This suggests that in general the set $\{ \varphi_i \mid i \in I \}$  
%of morphisms in $\CF$ appearing in an non-virtual characteristic $(S,S)$-biset 
%$$
%\Omega=\coprod_{i \in I} S \times_{(P_i,\varphi_i)} S,
%$$
%e.g., the one built by Broto-Levi-Oliver, plays  the role of ``a set of $S$-$S$-double cosets representatives of the  
%fusion system $\CF$'', and the formula \eqref{transfer} can be viewed as a ``Mackey decomposition formula over $S$-$S$-double cosets in $\CF $'' for the map $H^*(\Omega;A)$.
\end{exm}

%Assume $A$ is an abelian $p$-group of exponent $p^e$ and $\Omega$ satisfies \eqref{components}, \eqref{invariant} and
%\begin{enumerate}
%\renewcommand{\theenumi}{\alph{enumi}'}
%\setcounter{enumi}{2}
%\item $|\Omega| / |S|\equiv 1 \mod e$.  \label{coprime''}
%\end{enumerate}

%Then the argument in \cite[Proposition 5.5]{BLO2003theory} shows that $H^*(\Omega;A)$ is an idempotent in $\End(H^*(S;A))$ and that the image of $H^*(\Omega;A)$ is exactly 
%$$
%\Img H^*(\Omega;A)=H^*(\CF;A):=\{z\in H^*(S;A) | \varphi^*(z)=r^S_P(z) \text{ for all $\varphi\in \Hom_\CF(P,S)$}\}.
%$$

Assume that $\Omega$ is a characteristic element for $\CF$ and $A$ is an abelian $p$-group with trivial $S$-action. The argument in \cite[Proposition 5.5]{BLO2003theory} shows that $H^*(\Omega;A)$ is an idempotent in $\End(H^*(S;A))$ up to multiplication by the $p'$-number $\epsilon(\Omega)$ and that the image of $H^*(\Omega;A)$ is exactly 
$$
H^*(\CF;A):=\{z\in H^*(S;A) \mid \varphi^*(z)=\res^S_P(z) \text{ for all } \varphi\in \Hom_\CF(P,S)\}.
$$

Hence, given characteristic elements $\Omega$ and $\Lambda$ for $\CF$, $H^*(\Omega;A)$ and $H^*(\Lambda;A)$ are projections (up to a $p'$-factor) in $\End(H^*(S;A))$ that have the same image. The following corollary shows that, indeed, they only differ by a $p'$-factor.  %we prove below that $H^*(\Omega;A)=H^*(\Lambda;A)$.

%\begin{lem}\label{lemma_abcequal}
%Let $\CF$ be a saturated fusion system on the $p$-group $S$ and let $A$ be a abelian group of exponent $e$. If $\Omega$, $\Lambda \in A(S,S)$ satisfy \eqref{components}, \eqref{invariant} and \eqref{coprime''} then $H^*(\Omega;A)=H^*(\Lambda;A)$.
%\end{lem}
%\begin{proof}
%Let $\Omega$ and $\Lambda$ be as in the statement. According to \cite[Proposition 4.9]{Ragnarsson2006spectra} there is a large enough positive integer $k$ such that $\Lambda^k-\Omega^k=e\Upsilon$ for some $\Upsilon\in A(S,S)$. As $H^*(\cdot;A)$ is a ring homomorphism we have $H^*(\Lambda^k-\Omega^k;A)=H^*(\Lambda;A)^k-H^*(\Omega;A)^k$. Because both $H^*(\Lambda;A)$ and $H^*(\Omega;A)$ are idempotents we have $H^*(\Lambda;A)^k=H^*(\Lambda;A)$ and $H^*(\Omega;A)^k=H^*(\Omega;A)$. On the other hand, $e$ is the exponent of $A$ and therefore $H^*(e\Upsilon;A)=eH^*(\Upsilon;A)=0$.
%\end{proof}

%It is clear that if $\Omega\in A(S,S)$ satisfies  \eqref{components}, \eqref{invariant} and \eqref{coprime} then a disjoint union of an appropiate $p'$-number of copies of $\Omega$ satisfies \eqref{components}, \eqref{invariant} and \eqref{coprime''}. Hence we have:

\begin{cor}\label{corollary_[omega]uptoqnumber}
Let $\CF$ be a saturated fusion system on a finite $p$-group $S$ and let $A$ be an abelian $p$-group with trivial $S$-action. If $\Omega$ and $\Lambda$ are characteristics elements for $\CF$ then there is a $p'$-number $r$ such that $H^*(\Omega;A)=r\cdot H^*(\Lambda;A)$.
\end{cor}

\begin{proof}
After multiplying by suitable $p'$-numbers, we may assume that $\Omega$ and $\Lambda$ lie in $A(S,S)$.  Let $p^e$ be the exponent of $A$. As remarked after Definition~\ref{def_charidem}, there is a large enough positive integer $k$ such that $\Lambda^k-\Omega^k=p^e\Upsilon$ for some $\Upsilon\in A(S,S)$.  Because both $H^*(\Lambda;A)$ and $H^*(\Omega;A)$ are idempotents up to a $p'$-factor, we get $H^*(\Lambda;A)^k=q_1\cdot H^*(\Lambda;A)$ and $H^*(\Omega;A)^k=q_2\cdot H^*(\Omega;A)$, where $q_1$ and $q_2$ are $p'$-numbers. On the other hand, $p^e$ is the exponent of $A$ and therefore $H^*(p^e\Upsilon;A)=p^eH^*(\Upsilon;A)=0$. As $H^*(-;A)$ is a ring homomorphism we finally obtain
$$
0=H^*(\Lambda^k-\Omega^k;A)=H^*(\Lambda;A)^k-H^*(\Omega;A)^k=q_1\cdot H^*(\Lambda;A)-q_2\cdot H^*(\Omega;A).
$$
\end{proof}

%As mentioned above, the characteristic idempotent $\omega\in\pcomp{A(S,S)}$ is the limit of a subsequence of the sequence $\Omega$, $\Omega^2$, $\Omega^3$, $\ldots$, where $\Omega\in A(S,S)$ is any (virtual) characteristic biset. Similar arguments to those of the previous lemma and corollary shows that 

%\begin{lem}\label{lemma_abcequal_ci}
%Let $\CF$ be a saturated fusion system on the $p$-group $S$ and let $A$ be a $\plocz$-module of exponent $e$, e.g., a finite abelian $p$-group. If $\Omega$ is a characteristic biset for $\CF$ and $\omega$ is a characteristic idempotent for $\CF$ then there is a $p'$-number $q$ such that $H^*(\Omega;A)=q\cdot H^*(\omega;A)$.
%\end{lem}

We are now ready to define the transfer map. Working in degree $1$, we identify $H^1(S;A)=\Hom(S,A)$ and note that $H^1(\CF;A)=\Hom(S/[S,\CF],A)$.

\begin{dfn}\label{definition_transfer}
Let $\CF$ be a saturated fusion system on a finite $p$-group $S$ and let $\CH$ be a saturated fusion subsystem of $\CF$ on $S$. Set $A = S /[S,\CH]$ and consider the canonical projection $\pi \colon S \to S/[S,\CH]$. Given a characteristic element $\Omega$ for $\CF$, the \emph{transfer map from $\CH$ to $\CF$ with respect to $\Omega$} is 
\[
        \tau^{\CF}_{\CH,\Omega} = H^1(\Omega;A)(\pi) \colon S \to S/[S,\CH].
\]
\end{dfn}

When $\CH$ is the trivial fusion system $\CF_S(S)$ on $S$ then $[S,\CH]=[S,S]=S'$, the derived subgroup of $S$. In this case we write $\tau^{\CF}_{S,\Omega}$ instead of $\tau^{\CF}_{\CH,\Omega}$ and we call it the \emph{transfer map from $S$ to $\CF$ (with respect to $\Omega$)}. The transfer $\tau^{\CF}_{S,\Omega}$ was successfully used in \cite{DGMPtransfer} by three of the authors and Nadia Mazza to study control of transfer and weak closure in fusion systems.  In the next lemma, we show that if $\Sigma$ is another characteristic element for $\CF$ then $\tau_{\CH,\Omega}^\CF$ and $\tau_{\CH,\Sigma}^\CF$ only differ by the multiplication by a $p'$-number.

 \begin{lem}\label{lemma_tdependence}
Let $\CF$ be a saturated fusion system on a finite $p$-group $S$ and let $\CH$ be a saturated fusion subsystem of $\CF$ on $S$. Let $\Sigma$ and $\Omega$ be characteristic elements for $\CF$.
\begin{enumerate}
\item $\tau^{\CF}_{\CH,\Sigma}=r\cdot \tau^{\CF}_{\CH,\Omega}$ for some $p'$-number $r$. \label{lemma_dependence1}
\item $\Img(\tau^{\CF}_{\CH,\Sigma})=\Img(\tau^{\CF}_{\CH,\Omega})$.\label{lemma_dependence2}
\item $\Ker(\tau^{\CF}_{\CH,\Sigma})=\Ker(\tau^{\CF}_{\CH,\Omega})=[S,\CF]$.\label{lemma_dependence3} In particular, $\tau^{\CF}_{\CH,\Omega}$ can be viewed as a map from $S/[S,\CH]$ to itself.
\item $\tau^{\CF}_{\CH,\Omega}\circ \tau^{\CF}_{\CH,\Omega}=\epsilon(\Omega)\cdot \tau^{\CF}_{\CH,\Omega}$.   \label{lemma_almost_idempotent}
\end{enumerate}
\end{lem}
\begin{proof}
Statement \eqref{lemma_dependence1} follows immediately from Corollary \ref{corollary_[omega]uptoqnumber} and the definition of the transfer while (\ref{lemma_almost_idempotent}) reflects the fact that $H^*(\Omega;A)$ is an idempotent up to multiplication by the $p'$-number $\epsilon(\Omega)$. From \eqref{lemma_dependence1} we obtain \eqref{lemma_dependence2} and the first equality in \eqref{lemma_dependence3}. To simplify notation, in the rest of the proof we write $\tau^{\CF}_{\CH}$ instead of $\tau^{\CF}_{\CH,\Omega}$. Note that $[S,\CF]$ is contained in the kernel of $\tau^{\CF}_{\CH}$ because $\tau^{\CF}_{\CH}\in H^1(\CF;S /[S,\CH])$. To prove that $\Ker(\tau^{\CF}_{\CH})$ is not larger than $[S,\CF]$ we take $\Omega$ to be a characteristic biset for $\CF$; $\Omega$ then has the form $\Omega =  \coprod_{i \in I} S \times_{(P_i,\varphi_i)} S$ and
\begin{equation}
       \tau^{\CF}_{\CH} = \sum_{i \in I} \tr^{S}_{P_i} (\pi \circ \varphi_i). \label{cotransfer}
\end{equation}
For $x\in S$ we have
\begin{align*}
        \tau^{\CF}_{\CH}(x)
        &= \sum_{i \in I} \tr^{S}_{P_i}(\pi \circ \varphi_i)(x) \\
        &= \sum_{i \in I} \sum_{\tr\in [S/P_i]} (\pi \circ \varphi_i)((x\cdot t)^{-1}xt),
\end{align*}
where $[S/P_i]$ denotes a set of representatives of the left cosets of $P_i$ in $S$, and for $t\in [S/P_i]$, $x\cdot t$ is the unique element in $[S/P_i]$ such that $(x\cdot t)P_i=xtP_i$. Considering a set $W$ of $\langle x \rangle$-orbit representatives of $[S/P_i]$, we obtain

\begin{align*}
        \tau^{\CF}_{\CH}(x)
        &= \sum_{i \in I} \sum_{w\in W} \pi(\varphi_i(w^{-1}x^{r(w)}w)),
\end{align*}
where $r(w)$ denotes the length of the $\langle x \rangle$-orbit of $[S/P_i]$ containing $w \in W$.
As $\varphi_i(w^{-1}x^{r(w)}w)[S,\CF]=w^{-1}x^{r(w)}w[S,\CF]$, we find that
\begin{align*}
        \tau^{\CF}_{\CH}(x)+\pi([S,\CF])
        &= \pi\left(\sum_{i \in I} \sum_{w\in W} x^{r(w)}\right)+\pi([S,\CF])\\
 				&= \pi\left(\sum_{i \in I} x^{|S:P_i|}\right)+\pi([S,\CF])\\
				&= \pi\left(x^{|\Omega|/|S|}\right)+\pi([S,\CF])\\
				&= {|\Omega|/|S|}\cdot \pi(x)+\pi([S,\CF]).
\end{align*}				
If $x\in\Ker(\tau^{\CF}_{\CH})$, then $\tau^{\CF}_{\CH}(x)=0$ in $S/[S,\CH]$ and ${|\Omega|/|S|}\cdot \pi(x)\in \pi([S,\CF])$. Since $|\Omega|/|S|$ is a $p'$-number, %\marginpar{We need to introduce this fact somewhere} 
also $\pi(x)\in \pi([S,\CF])$. As $[S,\CH]\leq [S,\CF]$, we conclude that $x \in [S,\CF]$.
\end{proof}

Throughout the paper, in general, we will use the notation $\tau^{\CF}_{\CH}$ for the transfer map from $\CF$ to $\CH$ without specifying the characteristic elements.  By Lemma \ref{lemma_tdependence}, changing the characteristic element amounts to multiplying the transfer map by some $p'$-number, and does not change its kernel and image.

\begin{prp}\label{abelpfactordecomp}
Let $\CF$ be a saturated fusion system on a finite $p$-group $S$. If $\CH$ is a saturated fusion subsystem of $\CF$ on $S$, then
\[
S/[S, \CH] = [S, \CF]/[S, \CH] \times T_{\CF}[S,\CH]/[S,\CH],
\]
where $T_{\CF}$ denotes the subgroup of $S$ containing $S'$ and such that $T_{\CF}/S' = \Img(\tau_S^{\CF})$. In particular, $S/[S,\CF]$ is a direct factor of $S/[S,\CH]$.
\end{prp}

\begin{proof}
Applying part (3) of Lemma \ref{lemma_tdependence} to $\tau^{\CF}_{S}$ gives the equality
\[
S/S' = [S,\CF]/S' \times T_{\CF}/S'
\]
%from which we get the following diagram:
%\[
%\xymatrix@! @=-1pc{
%& S \ar@{-}[dddl] \ar@{-}[dr] & & \\
%& & T_{\CF}[S,\CH] \ar@{-}[dr] \ar@{-}[dddl] & \\
%& & & T_{\CF} \ar@{-}[dddl] \\
%[S,\CF] \ar@{-}[dr] & & & \\
%& [S, \CH] \ar@{-}[dr] & & \\
%&  & S' &
%}
%\]
%$for any saturated subsystem $\CH$ on $S$.
Factoring this equality by $[S,\CH]/S'$ gets us the result in the proposition.
\end{proof}

As a cyclic $p$-group has no proper nontrivial direct factors, the previous proposition immediately gives the following corollary.

\begin{cor}\label{cyclicpfac}
Let $\CF$ be a saturated fusion system on a finite $p$-group $S$ such that $[S, \CF] < S$ and let $\CH$ be a saturated fusion subsystem of $\CF$ on $S$. If $S/[S, \CH]$ is cyclic, then $\CH$ controls transfer in $\CF$, i.e., $[S, \CH] = [S, \CF]$.
\end{cor}

%*****************************************************************************************************************************
%*****************************************************************************************************************************
%*****************************************************************************************************************************

\section{Yoshida's theorem}\label{S:Yoshida}
%In the theory of finite groups, transfer plays a key role in finding proper nontrivial normal subgroups. It is a weaker form of fusion: that is, if a subgroup controls fusion, then it controls transfer. An interesting result on control of transfer appeared in \cite[Theorem 4.2]{Yoshida1978} where the author proves that if $G$ is a finite group with Sylow $p$-subgroup $S$, then $\N_G(S)$ controls transfer unless $C_p \wr C_p$ is a quotient of $S$. We will generalize this result  to fusion systems following closely the proof appearing in \cite[\S I.6]{Glauberman1977factor}.

In this section, we prove that for a saturated fusion system $\CF$ on a finite $p$-group $S$, if $C_p \wr C_p$ is not a homomorphic image of $S$, then the focal subgroups of $\CF$ and $\N_\CF(S)$ coincide. First, we recall a useful lemma that helps detect a homomorphic image isomorphic to $C_p \wr C_p$. This appears as Lemma 6.4 in \cite{Glauberman1977factor}.

\begin{lem} \label{T:wreath}
Let $R$ be a finite $p$-group having an elementary abelian subgroup $E 
$ of index $p$.  Suppose that there are $x \in E$ and $z \in R - E$  
such that
\[
        \prod_{i=0}^{p-1} z^{-i}xz^i \neq 1.
\]
Then $R$ has $C_p \wr C_p$ as a homomorphic image.
\end{lem}

We can now prove Theorem Y, using part (2) in Theorem T as a definition for the control of transfer.
 
\begin{thmY}
Let $\CF$ be a saturated fusion system on a finite $p$-group $S$ and let $\CH =  
N_\CF(S)$.  If $[S,\CF] \neq [S,\CH]$, then $S$ has $C_p \wr C_p$ as a  
homomorphic image.
\end{thmY}
\begin{proof}
Fix a characteristic biset $\Omega$ for $\CF$ and write
\[
	\Omega = \sum_{i \in I} [P_i,\varphi_i].
\]
Let $I_0 = \{ i\in I \mid P_i = S \}$.  Then 
$$ 
	\epsilon(\Omega) = \sum_{i\in I} |S:P_i| = |I_0| + \sum_{i\in I - I_0} |S:P_i| \equiv |I_0| \pmod p.
$$
By part \eqref{coprime} of Definition~\ref{def_charbiset}, it follows that $|I_0| \not\equiv 0 \pmod p$.

Suppose that $[S,\CF] \neq [S,\CH]$.  By Lemma~\ref{lemma_tdependence}, the transfer map $\tau_{\CH}^{\CF}$ has kernel $[S,\CF]$ and hence induces a nonsurjective map
\[
        \overline{\tau}^{\CF}_{\CH} \colon S/[S,\CF] \to S/[S,\CH].
\]
Let $\Img (\tau^{\CF}_{\CH}) = X/[S,\CH]$ where $ 
[S,\CH] \leq X < S$.  Take a maximal subgroup $Y$ of $S$ containing $X 
$, and take an element $x \in S - Y$ of minimal order.  We have
\begin{align*}
        \tau^{\CF}_{\CH}(x)
        &= \sum_{i \in I_0} \tr^{S}_{S}(\pi \circ \varphi_i)(x) + \sum_{j \in I  
- I_0} \tr^{S}_{P_j}(\pi \circ \varphi_j)(x) \\
        &= \sum_{i \in I_0} \varphi_i(x)[S,\CH] + \sum_{j \in I - I_0}  
\tr^{S}_{P_j}(\pi \circ \varphi_j)(x) \in Y/[S,\CH].
\end{align*}
Also, since $\varphi_i \in \Aut_{\CH}(S)$ whenever $i \in I_0$, 
\begin{align*}
        \sum_{i \in I_0} \varphi_i(x)[S,\CH] & = \sum_{i \in I_0}  
xx^{-1}\varphi_i(x)[S,\CH] \\
        & = x|I_0|[S,\CH] \notin Y/[S,\CH],
\end{align*}
because $x \notin Y$ and $|I_0|$ is not divisible by $p$.  
Thus, there is a proper subgroup $P < S$ and $ 
\varphi \in \Hom_\CF(P,S)$ such  
that
\begin{equation} \label{notcontained}
        \tr^{S}_{P}(\pi\circ\varphi)(x) \notin Y/[S,\CH].
\end{equation}
Note that for every $u \in S$,
\[
        \tr^{S}_{P}(\pi\circ\varphi)(u) = \sum_{t \in [S/P]} (\pi\circ\varphi) 
((u \cdot t)^{-1}ut) \in \varphi(P)[S,\CH]/[S,\CH].\]
Therefore, we can view $\tr^{S}_{P}(\pi\circ\varphi)$ as a map from $S$ to $Q/ 
[S,\CH]        $ where $Q=\varphi(P)[S,\CH]$.  By \eqref{notcontained}, we  
have $Q \nleq Y$ and hence $M := Y \cap Q < Q$.  Since $|S:Y|=p$, it  
follows that
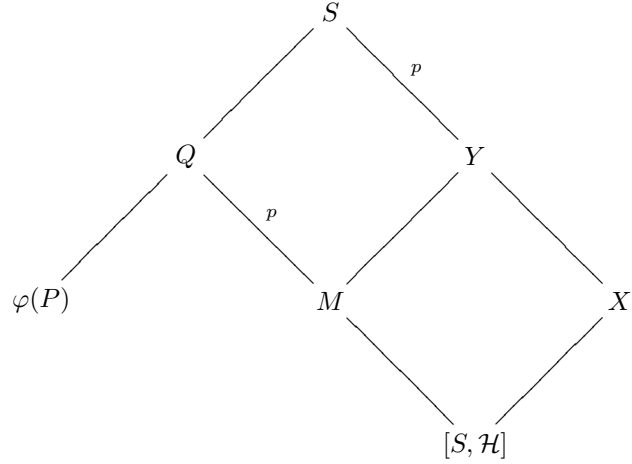
\begin{equation} \label{index p}
        |Q : M|=p.
\end{equation}

\[
\xymatrix@!{
          &                                      & S
\ar@{-}[dl]\ar@{-}[dr]^{p} \\
          & Q \ar@{-}[dl]\ar@{-}[dr]^{p}         &                       
       & Y \ar@{-}[dl]\ar@{-}[dr]\\
\varphi(P) &                                      & M \ar@{-}[dr]         
       &                            & X \ar@{-}[dl]\\
          &            &            & [S,\CH]\\
}
\]

Let $A$ be a maximal subgroup of $S$ containing $P$. We show that $x\in A$. 
Suppose $x \notin A$.  Then we can take $[S/A] = \{ x^i \mid 0 \leq i \leq p-1 \}$ and
\[
        x \cdot x^i =
\begin{cases}
        x^{i+1}        &\text{if $i < p-1$,}\\
        1                &\text{if $i = p-1$.}        
\end{cases}
\]
Using the transitivity of the transfer maps we get
\begin{align*}
        \tr^{S}_{P}(\pi\circ\varphi)(x)
        &= \tr^{S}_{A}(t^{A}_{P}(\pi\circ\varphi))(x)\\
        &= \sum_{i=0}^{p-1} \tr^{A}_{P}(\pi\circ\varphi) ((x \cdot  
x^i)^{-1}xx^i)\\
        &= \tr^{A}_{P}(\pi\circ\varphi)(x^p)\\
        &= \sum_{v \in [A/P]} (\pi\circ\varphi)((x^p \cdot        v)^{-1}x^pv)\\
        &= \sum_{w \in W} (\pi\circ\varphi)(w^{-1}x^{p\cdot r(w)}w)\\
        &\notin Y/[S,\CH]
\end{align*}
where $W$ denotes a set of $\langle x^p \rangle$-orbit representatives  
of $[A/P]$ and $r(w)$ denotes the length of the $\langle x^p \rangle$- 
orbit containing $w \in W$.  So there is a $w \in W$ such that $ 
\varphi(w^{-1}x^{p\cdot r(w)}w) \notin Y$. But by the minimality of the  
order $o(x)$ of $x$, we get
\[
        o(x) \leq o(\varphi(w^{-1}x^{p\cdot r(w)}w)) = o(w^{-1}x^{p\cdot r(w)}w) =  
o(x^{p\cdot r(w)}) < o(x),
\]
a contradiction.  Thus $x \in A$.

If $z \in S - A$, then $[S/A] = \{ z^i \mid 0 \leq i \leq p-1 \}$  
and $x \cdot z^i = z^i$ for all $i$ because $x \in A$ and $A \lhd S$.  
Therefore,
\begin{align*}
        \tr^{S}_{P}(\pi\circ\varphi)(x)
        &= \tr^{S}_{A}(\tr^{A}_{P}(\pi\circ\varphi))(x)\\
        &= \sum_{i=0}^{p-1} \tr^{A}_{P}(\pi\circ\varphi) ((x \cdot  
z^i)^{-1}xz^i)\\
        &= \sum_{i=0}^{p-1} t^{A}_{P}(\pi\circ\varphi) (z^{-i}xz^i)\\
        &= \tr^{A}_{P}(\pi\circ\varphi)(\prod_{i=0}^{p-1} z^{-i}xz^i).
\end{align*}
Suppose $\prod_{i=0}^{p-1} z^{-i}xz^i \in \Phi(A)$.  Since $\Phi(A) =  
A^p[A,A]$, we have $\tr^{S}_{P}(\pi\circ\varphi)(x) \in \Phi(Q/[S,\CH]) 
$.  But by \eqref{index p}, we have $\Phi(Q/[S,\CH]) \leq M/[S,\CH]$.  
Thus $\tr^{S}_{P}(\pi\circ\varphi)(x) \in Y/[S,\CH]$, contradicting  
\eqref{notcontained}.  Hence
\[
        \prod_{i=0}^{p-1} z^{-i}xz^i \notin \Phi(A).
\]
Now, by Lemma~\ref{T:wreath} applied to $R = S/\Phi(A)$ and $E=A/\Phi(A) 
$, the wreath product $C_p \wr C_p$ is a homomorphic image of $S/\Phi(A)$ and hence of $S$.
\end{proof}

Recall that if $\CF$ is a fusion system on a finite $p$-group $S$ and $\CH$ is a subsystem of $\CF$, then we say that $\CH$ \textit{controls transfer} in $\CF$ if $[S, \CF] = [S, \CH]$. 

\begin{cor}\label{T:yoshidacor1}
Let $\CF$ be a saturated fusion system on a finite $p$-group $S$. If any of the following conditions hold, then $\N_\CF(S)$ controls transfer in $\CF$. 
\begin{enumerate}
\item $S$ has nilpotence class less than $p$;\label{Tp:yoshidacor1a}
\item The exponent of $S$ is less than or equal to $p$;\label{Tp:yoshidacor1b}
\item $S$ is a regular $p$-group; \label{Tp:yoshidacor1c}
\item $p$ is odd and $S$ is metacyclic.\label{Tp:yoshidacor1d} 
\end{enumerate}
\end{cor}

\begin{proof}
The first two conditions imply the result since $C_p \wr C_p$ has nilpotence class $p$ and contains an element of order $p^2$. The third statement is immediate since a regular $p$-group does not have a homomorphic image isomorphic to $C_p \wr C_p$ and the last statement follows from the third as every metacyclic $p$-group is regular if $p$ is odd (cf. \cite[Satz III.10.2]{HuppertEG}). 
\end{proof}

Note that the last statement of the corollary is also a consequence of \cite[Proposition 5.4]{Stancu2006} or \cite[Theorem 4.1]{DRV2007} and that it cannot be extended to $p = 2$ since $C_2 \wr C_2 \cong D_8$ is metacyclic. However, if $p = 2$ and $S$ is metacyclic and not homocyclic abelian, dihedral, semidihedral or generalized quaternion, then $\CF$ is trivial and the result holds (see \cite{CravenGlesser} for a complete classification of fusion systems on metacyclic $p$-groups).   Also, \ref{T:yoshidacor1}.\ref{Tp:yoshidacor1c} is considerably different than \ref{T:yoshidacor1}.\ref{Tp:yoshidacor1a} since a regular $p$-group can have an arbitrarily large nilpotence class.

\begin{thm}[Huppert]\label{T:huppert}
Let $p$ be an odd prime and let $\CF$ be a saturated fusion system on a finite $p$-group $S$. If $S$ is nonabelian and metacyclic, then $[S, \CF] < S$.

\end{thm}

\begin{proof}
By Corollary \ref{T:yoshidacor1}.\ref{Tp:yoshidacor1d}, we may assume that $\N_S(\CF) = \CF$. In this case, $\CF$ is constrained and so, by \cite[Proposition 4.3]{BCGLO2005}, $\CF = \CF_S(G)$ for some finite group $G$ with Sylow $p$-subgroup $S$. Thus, $[S, \CF] = [S, \CF_S(G)] < S$ by \cite[Hilfssatz IV.8.5]{HuppertEG}.
\end{proof}

%*************************************************
\section{$p$-power index transfer}\label{S:OpF}%**
%*************************************************

In this section, we prove important properties of characteristic idempotents and transfer maps for fusion systems that will be used in the next section to prove Theorem T. The elementary situation in group theory we want to mimic is the following: let $G$ be a finite group with two subgroups $S$ and $L$ such that $G=SL$, i.e.\ $G=\{ xy \mid x\in S, y\in L \}$, and let $N=S\cap L$. Then we have a bijection between left coset spaces
\[
	L/N \xrightarrow{\cong} G/S
\]
induced by the inclusion $L \hookrightarrow G$. As a consequence, we get a commutative diagram
\[
\xymatrix{
S\ar[r] & G\ar[r]^{t_{S}^{G}}& S/S'\\
N\ar[r]\ar[u] &L\ar[u]\ar[r]^{t_{N}^{L}}& N/N'\ar[u]_{\rho}.
}
\]
where $t_{S}^{G}$, $t_{N}^{L}$ are group transfer maps, $\rho$ is the map induced by the inclusion $N \hookrightarrow S$, and all other arrows are inclusions. Furthermore, if $S$ is a Sylow $p$-subgroup of $G$, $L\unlhd G$ and we denote $\CF=\CF_S(G)$ and $\CN=\CF_N(L)$ the outer rectangle in the above gives a commutative diagram
\[
\xymatrix{
S\ar[rr]^{\tau_{S,G}^{\CF}}&& S/S'\\
N\ar[rr]^{\tau_{N,L}^{\CN}}\ar[u]^{\mathrm{incl}}&& N/N'\ar[u]_{\rho}.
}	
\]
In particular, we have $\rho(\Img(\tau^{\CN}_{N,L}))\subseteq \Img(\tau^{\CF}_{S,G})$. This inclusion between images of transfers is the result we want to generalize to fusion systems. As cosets do not make sense in the fusion system setting we use an alternative approach to prove this inclusion in the group case.

Viewing $G$ as an $(S,S)$-biset and $L$ as an $(N,N)$-biset in the obvious way,there is an isomorphism of $(N,S)$-bisets
\[
	S \times_N L \xrightarrow{\cong} G
\]
induced by the product map $(x,y)\in S\times L \mapsto xy \in G$.  Note that $G$ and $L$ are characteristic bisets for the fusion systems $\CF$ and $\CN$, respectively. We can rewrite this isomorphism of $(N,S)$-bisets as the following equality in $A(N,S)$:
\[
	[N, \mathrm{incl}]^S_N \circ L =  G \circ [N, \mathrm{incl}]^S_N.
\]
%Notice that $[N, \mathrm{incl}]^S_N\in A(N,S)$ is represented the $(N,S)$-biset $S$.
We give below an analogous equality in terms of characteristic idempotents, valid for any saturated fusion system, whose proof was provided by K\'ari Ragnarsson through private communication. (See also \cite{RagnarssonPark}) As we show in Corollary \ref{cor_keyforTate}, this is enough to deduce the inclusion between the images of the transfers. We refer the reader to the appendix for the definition and properties of a saturated subsystem of $p$-power index. Recall that $\omega_\Ff$ is the characteristic idempotent of $\Ff$.

\begin{thm}\label{RelationCharIdempPPower}
Let $\CF$ be a saturated fusion system on the $p$-group $S$. If $N$ is a normal subgroup of $S$ containing $O^p_\CF(S)$ and $\CF_N$ is the unique saturated subsystem of $\CF$ on $N$ of $p$-power index, then
\begin{enumerate}
\item $[N, \mathrm{incl}]^S_N\circ \omega_{\Ff_N} = \omega_\Ff\circ [N, \mathrm{incl}]^S_N$;
\item $\omega_{\Ff_N}\circ [N, \id]^N_S = [N, \id]^N_S \circ \omega_\Ff$.
\end{enumerate}
\end{thm}

\begin{cor}\label{CharacteristicTransRes}
In the situation of Theorem~\ref{RelationCharIdempPPower}, we have
\begin{equation}
\tr^S_N \circ H^*(\omega_{\Ff_N};A) = H^*(\omega_\Ff;A) \circ \tr^S_N : H^*(N;A)\to H^*(\CF;A)\label{eqn4}
\end{equation}
\begin{equation}
H^*(\omega_{\Ff_N};A) \circ \res^S_N = \res^S_N \circ H^*(\omega_\Ff;A): H^*(S;A)\to H^*(\CF_N;A)\label{eqn3}
\end{equation}
for any $\ZZ_{(p)}S$-module $A$.
\end{cor}

\begin{proof}
This follows from Theorem~\ref{RelationCharIdempPPower} and from the equalities $H^*([N,\mathrm{incl}]_N^S)=\res^S_N$ and $H^*([N,\id]_S^N)=\tr_N^S$.
\end{proof}

\begin{cor}\label{cor_keyforTate}
In the situation of Theorem~\ref{RelationCharIdempPPower}, the diagram
\[
\xymatrix{
S\ar[rr]^{\tau_{S,\omega_\CF}^{\CF}}&& S/S'\\
N\ar[rr]^{\tau_{N,\omega_{\CF_N}}^{\CF_N}}\ar[u]^{\mathrm{incl}}&& N/N'\ar[u]_{\rho}
}
\]
where $\rho$ is the map induced by the inclusion $N \hookrightarrow S$, is commutative. In particular, we have 
\[
	\rho(\Img(\tau^{\CF_N}_N))\subseteq \Img(\tau^{\CF}_S).
\]
\end{cor}

\begin{proof}
By Corollary~\ref{CharacteristicTransRes}, we get the following commutative diagram
\[
\xymatrix{
H^1(S,S/S')\ar[rr]^{H^1(\omega_{\CF},S/S')}\ar[d]_{\res^S_N} &&H^1(S,S/S')\ar[d]^{\res^S_N}\\
H^1(N,S/S')\ar[rr]^{H^1(\omega_{\CF_N},S/S')} &&H^1(N,S/S')\\
H^1(N,N/N')\ar[rr]^{H^1(\omega_{\CF_N},N/N')}\ar[u]^{\rho_*} &&H^1(N,N/N')\ar[u]_{\rho_*}.
}
\]
For a group $H$, let $\pi_H\colon H\to H/H'$ denote the canonical surjection.  Since $\res^S_N(\pi_S)=\rho_*(\pi_N)$, chasing arrows gives
\[
	\tau_{S,\omega_\CF}^{\CF}\circ\mathrm{incl}_N^S = \rho\circ\tau_{N,\omega_{\CF_N}}^{\CF_N},
\]
as desired.
\end{proof}

Now we turn to the proof of Theorem~\ref{RelationCharIdempPPower}. First we need several lemmas.

\begin{lem}[\cite{Ragnarsson2006spectra}]\label{StabilityCharacterization}
Let $\omega_\Ff\in A(S,S)_{(p)}$ be the characteristic idempotent of a
saturated fusion system $\Ff$ on a finite $p$-group $S$. Let $T$ be a finite $p$-group and let $X\in A(S,T)_(p)$. The following are equivalent:
\begin{enumerate}
\item $X\circ\omega_\Ff = X$.
\item $X$ is right $\CF$-stable, in the sense that for all $P\le S$ and $\varphi\in\Hom_\Ff(P,S)$ we have
$X\circ [P,\varphi]^S_P = X\circ [P,\mathrm{incl}]^S_P$.
\end{enumerate}
\end{lem}

\begin{proof}
This is proved for stable maps in \cite[Corollary 6.4]{Ragnarsson2006spectra}, but the same argument
works for $X\in A(S, T)_{(p)}$.
\end{proof}

For the definition of \normal subsystem used in the next lemma, we refer the reader to Definition~\ref{dfn_Markusnormalsubsystem} in the appendix.

\begin{lem}[{\cite[Theorem 8.2]{RagnarssonStancu2009idempotents}}]\label{NormalCharacterization}
Let $\Ff$ be a saturated fusion system on a finite $p$-group $S$ and let $\Nn$ be a saturated
fusion subsystem of $\Ff$ on a strongly $\CF$-closed subgroup $N$ of $S$. 
Let $\omega_\Nn$ denote the characteristic idempotent of $\Nn$. The following are equivalent:
\begin{enumerate}
\item $\Nn$ is an \normal subsystem of $\Ff$.
\item For every subgroup $Q$ of $N$ and every morphism $\varphi\in\Hom_\Ff(Q,S)$, 
the following identity in $A(Q,Q)_{(p)}$ holds:
$$[\varphi(Q),\varphi^{-1}]^Q_N\circ\omega_\Nn\circ [Q,\varphi]^N_Q = [Q,\id]_N^Q\circ\omega_\Nn\circ[Q,\mathrm{incl}]_Q^N\,.$$
\end{enumerate}
\end{lem}

%Let $\CF$ be a saturated fusion system $\CF$ on a finite $p$-group $S$. Consider a normal subgroup $N\lhd S$ containing the hyperfocal subgroup %$O^p_\Ff(S)$. According to \cite[Theorem 4.3]{BCGLO2007extensions} there is a unique saturated subsystem $\CF_N$ $p$-power index of $\CF$ associated %to $N$. We prove in the appendix (Proposition \ref{prp_stronglyclosedplusnormality}) that, in this situation, $N$ is strongly $\CF$-closed and %$\CF_N$ is a normal subsystem.

%\begin{lem}\label{NormalFrattini}
%Let $\Ff$ be a saturated fusion system on a finite $p$-group $S$. Let
%$O^p_\Ff(S)\le N \unlhd S$, and let $\Ff_N$ be the saturated subsystem of $p$-power
%index of $\Ff$ on $N$. Given $P\le N$ and a morphism $\varphi$ in
%$\Hom_\Ff(P,S)$, there exist $t\in S$ and $\psi\in\Hom_{\Ff_N}(P,S)$ such that $\varphi=\psi\circ c_t$.
%\end{lem}

%\begin{proof}
%This is proved by a Frattini argument, using the fact that
%$\Aut_\Ff(N)/\Aut_{\Ff_N}(N)=S/N$.
%{\bf [How about writing up the proof as follows?]}
%Lemma \ref{lem:aschbacher} applied to the normal subsystem $\CF_N$ gives $\alpha \in \Aut_\CF(N)$ and $\psi_1 \in \Hom_{\CF_N}(\alpha(P),N)$ such %that $\varphi = \psi_1 \circ \alpha|_P$. Since $N$ is normal in $S$, it is fully $\CF$-normalized, and so $\Aut_\CF(N) = O^p(\Aut_\CF(N))\Aut_S(N)$. %But by definition of $\CF_N$, $O^p(\Aut_\CF(N))$ is contained in $\Aut_{\CF_N}(N)$. Thus $\alpha = \psi_2 \circ c_t$ for some $t\in S$ and $\psi_2 %\in \Aut_{\CF_N}(N)$. Set $\psi=\psi_1\circ\psi_2|_{c_t(P)}$.  Then $\psi \in \Hom_{\CF_N}(c_t(P),N)$ and $\varphi=\psi\circ c_t$.
%\end{proof}

\begin{proof}[Proof of Theorem~\ref{RelationCharIdempPPower}]
First we remark that parts (1) and (2) of the theorem are equivalent by applying the opposite homomorphism \cite[Definition 3.19]{RagnarssonStancu2009idempotents}. We proceed to prove part (1). Note that by Proposition \ref{normal p-subsystem}, $N$ is a strongly $\CF$-closed subgroup of $\CF$ and $\CF_N$ is a \normal subsystem of $\CF$. 

Since $\Ff_N$ is a subsystem of $\Ff$, the $\Ff$-stability of $\omega_\Ff$ implies
	that $\omega_\Ff\circ [N,\mathrm{incl}]^S_N$ is $\Ff_N$-stable. Hence, by Lemma~\ref{StabilityCharacterization}, we have
$$\omega_\Ff \circ [N, \mathrm{incl}]^S_N = \omega_\Ff \circ [N, \mathrm{incl}]^S_N \circ \omega_{\Ff_N}$$
and it suffices to show that 
$$\omega_\Ff \circ [N, \mathrm{incl}]^S_N \circ \omega_{\Ff_N} = [N, \mathrm{incl}]^S_N \circ \omega_{\Ff_N}\,,$$
which, by applying the opposite homomorphism, is equivalent to
$$\omega_{\Ff_N} \circ [N, \id]^N_S \circ \omega_\Ff = \omega_{\Ff_N} \circ [N, \id]^N_S\,.$$

We prove this last equation by showing that $\omega_{\Ff_N} \circ [N, \id]^N_S$ is right $\CF$-stable.
Now, for $P\le S$ and $\varphi\in\Hom_\Ff(P,S)$, the double coset formula gives 
$$[N, \id]^N_S \circ [P,\varphi]^S_P = \displaystyle\sum_{x\in N\backslash S/\varphi(P)} [\varphi^{-1}(\varphi(P)\cap N^x), c_x \circ \varphi]^N_P\,.$$
Since $N$ is normal in $S$ we have $N^x = N$, and since $N$ is strongly $\Ff$-closed we have
$\varphi(P)\cap N = \varphi(P\cap N)$, so the equation simplifies to
$$[N, \id]^N_S \circ [P, \varphi]^S_P = \displaystyle\sum_{x\in [N\backslash S/\varphi(P)]} [P\cap N, c_x \circ \varphi]^N_P\,.$$
Using Lemma \ref{normal p-subsystem} on $(\varphi|_{P\cap N})^{-1}$ we find $t\in S$ and $\psi\in\Hom_{\Ff_N}(P\cap N,N)$ such that $\varphi|_{P\cap N} = c_t \circ\psi$.

Using that $[N, c_x]^N_N\circ [N, c^{-1}_x]^N_N = [N, \id]^N_N$ 
in Lemma \ref{NormalCharacterization} for the \normal subsystem $\CF_N$ we get that, for all $x\in S$,
$\omega_{\Ff_N}\circ [N, c_x]^N_N = [N, c_x]^N_N\circ\omega_{\Ff_N}$. This result applied in the previous equation give
\begin{equation}
\begin{array}{rl}
\omega_{\Ff_N} \circ [N, \id]^N_S \circ [P,\varphi]^S_P &= \omega_{\Ff_N} \circ \displaystyle\sum_{x\in N\backslash S/\varphi(P)} [P\cap N, c_x \circ \varphi]^N_P\\
\quad &=\omega_{\Ff_N} \circ \displaystyle\sum_{x\in N\backslash S/\varphi(P)} [P\cap N, c_x \circ c_t \circ \psi]^N_P\\
\quad &=\omega_{\Ff_N} \circ \left(\displaystyle\sum_{x\in N\backslash S/\varphi(P)} [N, c_x \circ c_t]^N_N\right)\circ[P\cap N,\psi]^N_P\\
\quad &=\left(\displaystyle\sum_{x\in N\backslash S/\varphi(P)} [N, c_x \circ c_t]^N_N\right)\circ \omega_{\Ff_N} \circ[P\cap N,\psi]^N_P\\
\quad &=\left(\displaystyle\sum_{x\in N\backslash S/\varphi(P)} [N, c_x \circ c_t]^N_N\right)\circ \omega_{\Ff_N} \circ[P\cap N,\mathrm{incl}]^N_P\\
\quad &=\omega_{\Ff_N} \circ \left(\displaystyle\sum_{x\in N\backslash S/\varphi(P)} [N, c_x \circ c_t]^N_N\right)\circ[P\cap N,\mathrm{incl}]^N_P\\
\quad &= \omega_{\Ff_N} \circ \displaystyle\sum_{x\in N\backslash S/\varphi(P)} [P\cap N, c_x \circ c_t]^N_P\label{eqn1}\,.
\end{array}
\end{equation}
On the other hand, the double coset formula gives
\begin{equation}
\omega_{\Ff_N} \circ [N, \id]^N_S \circ [P, \mathrm{incl}]^S_P = \omega_{\Ff_N} \circ \displaystyle\sum_{y\in N\backslash S/P} [P\cap N, c_y]^N_P\label{eqn2},
\end{equation}
and the result follows by showing that the expressions in \eqref{eqn1} and \eqref{eqn2} are equal. 
To be able to compare the two sums we try to have the summation over the same indices. Consider the maps $\alpha, \beta : S \to A(P,N)$ 
defined by $\alpha(x) = \omega_{\Ff_N} \circ [P \circ N, c_x \circ c_t]^N_P$ and $\beta(y) = \omega_{\Ff_N} \circ [P \circ N, c_y]^N_P$.
For $y\in S$, the map $\beta$ is constant on the double coset $NyP$. Since $N$ is normal in $S$, $NP$ is the subgroup of $S$ and 
we have $NyP = yNP$. Hence 
$$\omega_{\Ff_N} \circ \displaystyle\sum_{x\in N\backslash S/P} [P\cap N, c_y]^N_P=\frac 1{|NP|}\displaystyle\sum_{y\in S} \beta(y)\,.$$

For $x\in S$, reversing the algebraic manipulations leading to \eqref{eqn1} we obtain that
$\alpha(x) =\omega_{\Ff_N} \circ [P\cap N, c_x \circ \varphi]^N_P$. Thus, $\alpha$ 
is constant on the double coset $Nx\varphi(P)$, and we get
$$\omega_{\Ff_N} \circ \displaystyle\sum_{x\in N\backslash S/\varphi(P)} [P\cap N, c_x \circ c_t]^N_P=\frac 1{|N\varphi(P)|}\displaystyle\sum_{x\in S}\alpha(x)\,.$$
Observe that $\beta(xt) = \alpha(x)$ for all $x\in S$, so $\displaystyle\sum_{y\in S} \beta(y)=\displaystyle\sum_{x\in S}\alpha(x)$. 
Moreover $|NP| = |N\varphi(P)|$ since $N$ is strongly $\Ff$-closed. We conclude that
$$\omega_{\Ff_N} \circ [N, \id]^N_S \circ [P,\varphi]^S_P = \omega_{\Ff_N} \circ [N, \id]^N_S \circ [P, \mathrm{incl}]^S_P\,.$$ 
This shows that $\omega_{\Ff_N} \circ [N, \id]^N_S$ is $\Ff$-stable, completing the proof.
\end{proof}

%*****************************************************************************************************************************
%*****************************************************************************************************************************
%*****************************************************************************************************************************

\section{Tate's theorem}
\label{S:Tate}

Recall that for a saturated fusion system $\CF$ on a finite $p$-group $S$, $T_\CF$ denote the subgroup of $S$ containing $S'$ such that $T_\CF/S'=\Img(\tau^{\CF}_{S})$ and that by Proposition \ref{abelpfactordecomp} we get
\begin{equation}
S/S'=[S,\CF]/S'\times T_\CF/S'. \label{lemma_transfer_direct}
\end{equation}

\begin{prp}\label{normalT}
Let $\Ff$ be a saturated fusion system on $S$ and $\Nn$ an \normal subsystem of $\Ff$ on a strongly $\CF$-closed subgroup $N$ of $S$.  For every $s\in S$, 
\[
	c_s \circ \tau_{N,\omega_{\CF_N}}^{\CF_N} \circ c_{s^{-1}} = \tau_{N,\omega_{\CF_N}}^{\CF_N}.
\]
In particular, $T_{\CF_N} \unlhd S$.
\end{prp}

\begin{proof}
This follows immediately from Lemma~\ref{NormalCharacterization} and the definition of $\tau_{N,\omega_{\CF_N}}^{\CF_N}$.
\end{proof}

\begin{prp}\label{TfTfn}
Let $\CF$ be a saturated fusion system on a finite $p$-group $S$,. If $[S,\CF]\leq N\leq S$, then $T_\CF \cap N=T_{\CF_N}S'$. 
\end{prp}

%\begin{proof}
%Clearly $T_{\CF_N}S' \leq T_\CF \cap N$, so it suffices to show that $T_\CF \cap N \leq T_{\CF_N}S'$. From (\ref{lemma_transfer_direct}) we have $N/N' \cong [N,\CF_N]/N' \times T_{\CF_N}/N'$; in particular, $N=T_{\CF_N}[N,\CF_N]$. From \eqref{lemma_transfer_direct} again we get $S/S' \cong [S,\CF]/S' \times T_\CF/S'$, and hence $T_\CF \cap [S,\CF]=S'$.  Corollary \ref{cor_keyforTate} gives $T_{\CF_N}\leq T_\CF$, hence we have $T_\CF \cap N = T_\CF \cap T_{\CF_N}[N,\CF_N]=T_{\CF_N}(T_\CF \cap [N,\CF_N])\leq T_{\CF_N}S'$, by Dedekind's lemma.
%\end{proof}

\begin{proof}
By Corollary \ref{cor_keyforTate}, $T_{\CF_N}\leq T_\CF$. Using (\ref{lemma_transfer_direct}) for both $S$ and $N$ we get $T_{\CF} \cap [S,\CF] = S'$ and $N = T_{\CF_N}[N,\CF_N]$. Dedekind's lemma then gives:
\begin{align*}
T_{\CF_N}S' \leq T_{\CF} \cap N = T_{\CF} \cap T_{\CF_N}[N, \CF_N] & = T_{\CF_N}(T_{\CF} \cap [N, \CF_N]) \\ & \leq T_{\CF_N}(T_{\CF} \cap [S, \CF]) = T_{\CF_N}S'.
\end{align*} 
\end{proof}

\noindent The following gives a crucial inductive argument.

\begin{prp}\label{T:induction}
Let $\CF$ be a saturated fusion system on a finite $p$-group $S$, and let $O^p_{\CF}(S)\leq U\unlhd S$.  If $T_{\CF_U}[U,S]\leq V\leq U$, then $S/V\cong U/V \times T_{\CF}V/V$.
\end{prp}

\begin{proof}
The hypotheses imply that $[V,S]\leq [U,S]\leq V$. So, $V\unlhd S$ and $U/V\leq Z(S/V)$. Moreover, by (\ref{lemma_transfer_direct}),
\[
S = [S,\CF]T_{\CF} = O^p_{\CF}(S)S'T_{\CF} = O^p_{\CF}(S)T_{\CF} \leq UT_{\CF} = UT_{\CF}V \leq S
\]
and hence $S = U\cdot(T_{\CF}V)$. It remains to show that $T_{\CF}V \cap U= V$. By Dedekind's lemma, $(T_{\CF}V) \cap U = (T_{\CF} \cap U)V$, and hence it is also equivalent to $T_\CF \cap U \leq V$.  We proceed by induction on $|S:U|$. The case $U=S$ being trivial, we assume $U < S$.  Choose a subgroup $W$ of index $p$ in $S$ and containing $U$. As $S/W$ is abelian, $W$ contains $S'$ and, hence, it contains $[S,\CF]$. Since $(\CF_W)_U=\CF_U$ and $O^p_{\CF_W}(W)=O^p_\CF(S)$ by Corollary \ref{cor:samehyperfocal}, we have $O^p_{\CF_W}(W)\leq U\unlhd W$ and $T_{(\CF_W)_U}[U,W]\leq V\leq U$.  By induction, it follows that $U \cap T_{\CF_W}V=V$.  By Proposition~\ref{normal p-subsystem} and Proposition~\ref{normalT}, we have $T_{\CF_W}V\unlhd S$. Let $\ol{\ \cdot\ }$ denote the image modulo $T_{\CF_W}V$.  Since $U/V\leq \Z(S/V)$, we have $\ol{W} \leq \Z(\ol{S})$.  Thus, $\ol{S}/\Z(\ol{S}) \cong (\ol{S}/\ol{W})/(\Z(\ol{S}/\ol{W})$ is cyclic, and hence $\ol{S}$ is abelian.  Therefore
  $S'\leq T_{\CF_W}V$ and so by Proposition~\ref{TfTfn}, $T_\CF \cap W= T_{\CF_W}S'\leq T_{\CF_W}V$. So $T_\CF \cap U = T_\CF \cap W \cap U \leq (T_{\CF_W}V) \cap U = V$, as desired.
\end{proof}

We are now ready to prove Theorem T.

\begin{thmT}
Let $\CF$ be a saturated fusion system on a finite $p$-group $S$, and let $\CH$ be a saturated fusion subsystem of $\CF$ on $S$. The following are equivalent.
\begin{enumerate}
\item $E^p_\CF(S)=E^p_\CH(S)$.
\item $A^p_\CF(S)=A^p_\CH(S)$.
\item $O^p_\CF(S)=O^p_\CH(S)$.
\end{enumerate}
\end{thmT}

\begin{proof}
(3) $\Rightarrow$ (2) $\Rightarrow$ (1): Follows from Corollary~\ref{cor:focalhyperfocal}.

(1) $\Rightarrow$ (3): Suppose (1). Then $O^p_{\CF}(S)\leq \Phi(S)O^p_{\CF}(S)=\Phi(S)O^p_{\CH}(S)$. Applying Proposition~\ref{T:induction} to $U=O^p_{\CF}(S)$ and $V=O^p_{\CH}(S)[O^p_{\CF}(S),\CH]$, we get $S/V\cong U/V \times T_{\CF}V/V$. Since $U/V\leq \Phi(S)/V=\Phi(S/V)$, it follows that $S/V=T_\CF V/V$, and hence $U/V=1$, that is, $O^p_\CF(S)=O^p_\CH(S)[O^p_\CF(S),\CH]$. Let $\ol{\CH}=\CH/O^p_\CH(S)$ and $\ol{S}=S/O^p_\CH(S)$.  Then $\ol{\CH}=\CF_{\ol{S}}(\ol{S})$, and hence $\ol{O^p_\CF(S)}=[\ol{O^p_\CF(S)},\ol{\CH}]=[\ol{O^p_\CF(S)},\ol{S}]$.  Since $\ol{S}$ is a finite $p$-group, it follows that $\ol{O^p_\CF(S)}=1$, as desired.
\end{proof}

\begin{proof}[Proof of Corollary \ref{intro_cor_p-nilpotency}] 
From section \ref{S:biset}, $H^1(\CF;\FF_p)=\Hom(S/[S,\CF],\FF_p)$, and this is clearly isomorphic to the elementary abelian $p$-group $S/\Phi(S)[S,\CF]=S/E^p_\CF(S)$. For the trivial fusion system $\CF_S(S)$ we obtain $H^1(\CF_S(S);\FF_p)=\Hom(S/[S,S],\FF_p)=H^1(S;\FF_p)$, which is isomorphic to $S/\Phi(S)=S/E^p_{\CF_S(S)}$. From the hypothesis we get $E^p_{\CF_S(S)}=E^p_\CF(S)$, and then by Tate's theorem $O^p_\CF(S)=O^p_{\CF_S(S)}=\{1\}$. This can only be the case if $\CF=\CF_S(S)$.
\end{proof}

%*****************************************************************************************************************************
%*****************************************************************************************************************************
%*****************************************************************************************************************************
\appendix
\section{Invariant fusion systems}\label{S:appendix}

For the convenience of the reader, we recall definitions and some  
standard properties of $p$-power index subsystems, \normal subsystems and quotient systems used in this paper. 

%For basics of fusion systems we refer the reader to  \cite{BLO2003theory}, \cite{puig:frobeniussystems} and  \cite{puig:frobeniuscategories}.

%

%Fusion systems and their saturation axioms were introduced by Puig  
%\cite{puig:frobeniussystems}, \cite{puig:frobeniuscategories} under  
%the name of Frobenius systems. Broto, Levi and Oliver  
%\cite{BLO2003theory} developed this axiomatic approach and gave a  
%different, equivalent set of saturation axioms. The classical examples  
%of saturated fusion systems are the ones coming from the $p$-local  
%structure of a finite group or of a block algebra. We assume that the  
%reader is familiar with the fusion systems and their basic properties.

%Alperin's theorem on $p$-local control of fusion also holds for  
%saturated fusion systems \cite{BLO2003theory}, \cite{Stancu2006}. The  
%theorem roughly asserts that the $\CF$-essential and $\CF$-maximal  
%automorphisms suffice to determine the whole fusion system $\CF$.  
%Hence, in general, when proving properties for all morphisms in a  
%fusion system, one does it through restricting to specific groups of  
%automorphisms. We use this approach later in the Appendix.

In the proofs of our transfer theorems we deal with a special class of  
fusion subsystems containing the hyperfocal subgroup.

\begin{dfn}\cite[Definition 3.1]{BCGLO2007extensions}
Let $\CF$ be a saturated fusion system on a finite $p$-group $S$ and $ 
\Hh$ a fusion subsystem of $\CF$ on a
subgroup $T$ of $S$. We say that $\Hh$ is a {\it $p$-power index  
subsystem of $\CF$} if  $T$ contains
$O^p_{\CF}(S)$ and $\Aut_{\Hh}(P)$ contains $O^p(\Aut_\CF(P))$ for all subgroups $P$ of $T$.
\end{dfn}

\begin{thm}[{\cite[Theorem 4.3]{BCGLO2007extensions}}]\label{thm:p-powerindexbijection}
Let $\CF$ be a saturated fusion system on a finite $p$-group $S$. There is a bijection between the subgroups of $S$ containing $O^p_\CF(S)$ and the saturated $p$-power index subsystems of $\CF$.
\end{thm}

The above result was stated (with an additional hypothesis) independently by Puig in \cite[7.3]{puig:frobeniuscategories}. Let $O^p(\CF)$ denote the unique saturated subsystem of $\CF$ on $O^p_\CF(S)$ of $p$-power index and, more generally, let $\CF_U$ denote the unique saturated
fusion subsystem of of $\CF$ on $U$ of $p$-power index, for $O^p_{\CF}(S)  
\leq U \leq S$.

Our first goal is to see what more is true about $U$ and $\CF_U$ if $O^p_{\CF}(S) \leq U \unlhd S$. This requires us to recall the definition of an $\CF$-invariant subsystem.

\begin{dfn}\label{dfn_Markusnormalsubsystem}
Let $\CF$ be a fusion system on a finite $p$-group $S$ and $\Hh$ a  
fusion
subsystem of $\CF$ on a subgroup $T$ of $S$. We say that $\Hh$ is $\CF$-\textit{\normal} if $T$ is strongly $\CF$-closed and if for every isomorphism
$\varphi:Q\to P$ in $\CF$ and any two subgroups $U, V \leq Q\cap P$,
we have
$$\varphi\circ\Hom_{\Hh}(U,V)\circ\varphi^{-1}\subseteq
                        \Hom_{\Hh}(\varphi(U),\varphi(V))\,.$$
\end{dfn}

In the presence of saturation, there is a very useful characterization of  
\normal subsystems due to Puig \cite[6.6]{puig:frobeniuscategories}. Note that in \cite{Lincklemann}, Linckelmann calls saturated invariant subsystems \textit{normal}.

%Note
%that Aschbacher uses a different terminology, calling {\it $\CF$- 
%invariant} fusion subsystem what we call here normal fusion
%subsystem in $\CF$.

\begin{lem}[\cite{aschbacher},\cite{puig:frobeniuscategories}]\label{lem:aschbacher}
Let $\CF$ be a saturated fusion system on a finite $p$-group $S$. A saturated  
fusion subsystem $\CH$ on a strongly $\CF$-closed subgroup $T$ of $S$ is $\CF$-\normal if and only if the following  
conditions are satisfied
\begin{itemize}
\item[(i)] $\Aut_\CF(T)\leq\Aut(\Hh)$,
\item[(ii)]\label{lem:aschacher-ii} any morphism $\psi\in\Hom_\CF(P,Q)$ with $P,Q\le T$ decomposes
as $\psi=\phi\circ\chi|_{P}$ where $\phi\in\Hom_{\CH}(\chi(P),Q)$ and $\chi 
\in\Aut_\CF(T)$.
\end{itemize}
\end{lem}

We will use this to show that if $N$ is a normal subgroup of $S$ containing $O^p_{\CF}(S)$, then $\CF_N$ is $\CF$-invariant. First we will need the following lemma which will help us to prove the morphism decomposition component of Aschbacher's criterion. As in \cite[Definition 3.3]{BCGLO2007extensions}, let $O^p_*(\CF)$ denote  
the smallest restrictive subcategory of $\CF$ whose morphism set  
contains $O^p(\Aut_\CF(P))$ for all subgroups $P\leq S$. Using  
Alperin's fusion theorem and the fact that $\Aut_\CF(P)=O^p(\Aut_ 
\CF(P))\Aut_S(P)$ for any fully $\CF$-normalized subgroup $P$ of $S$, one obtains  
the following decomposition lemma where $c_s$ denotes the map induced by conjugation with an element $s$.

\begin{lem}{\cite[Lemma 3.4]{BCGLO2007extensions}}\label{lem:fist_decomp}
Let $\CF$ be a saturated fusion system on a finite $p$-group $S$. If $P \leq S$ and $\psi\in\Hom_\CF(P,S)$, then there exist $s \in S$ and  
$\varphi\in\Hom_{O^p_*(\CF)}(c_s(P),S)$ such that $\psi = \varphi \circ c_s| 
_P$.
\end{lem} 

We quickly mention the following useful corollary. 

\begin{cor}\label{cor:focalhyperfocal} %\marginpar{Can we offer more of a proof here?}
If $\CF$ is a saturated fusion system on a finite $p$-group $S$, then $A^p_{\CF}(S)=[S,S]O^p_{\CF}(S)$.
\end{cor}
\begin{proof}
With the notation in Lemma \ref{lem:fist_decomp} we have $[\psi,u]=[\varphi\circ c_s,u]=[\varphi,c_s(u)][c_s,u]\in [S,S]O^p_{\CF}(S)$.
\end{proof}

Using Lemma \ref{lem:fist_decomp}, we now show that normal subgroups containing the hyperfocal subgroup give rise to invariant subsystems.

\begin{prp} \label{normal p-subsystem}
Let $\CF$ be a saturated fusion system on $S$. If $N$ is a normal  
subgroup of $S$ containing $O^p_\CF(S)$, then
\begin{enumerate}
\item $N$ is strongly $\CF$-closed;
\item $\CF_N$ is a saturated $\CF$-\normal fusion subsystem.
%\item For every $P\leq N$ and $\varphi \in \Hom_\CF(P,S)$, there exist  
%$s\in S$ and $\psi \in \Hom_{\CF_N}(c_s(P),S)$ such that $\varphi =  
%\psi \circ c_s|_P$.\label{normal p-subsystem decompose phi}
\end{enumerate}
\end{prp}

\begin{proof}
Let $P\leq N$ and let $\psi \in \Hom_\CF(P, S)$.  By Lemma~ 
\ref{lem:fist_decomp}, there exist $s \in S$ and $\phi \in  
\Hom_{O^p_*(\CF)}(c_s(P),S)$ such that $\psi=\phi\circ c_s|_P$. If $u\in P$, then
\[
        \psi(u) = \phi(c_s(u))c_s(u)^{-1}c_s(u),
\]
where $\phi(c_s(u))c_s(u)^{-1} \in O^p_\CF(S)\leq N$ and $c_s(u) \in N 
$ because $N \unlhd S$.  Thus, $\psi(u) \in N$ and $N$  
is strongly $\CF$-closed, proving (1), from which it follows that $\phi 
$ belongs to $\CF_N$. Invoking Lemma \ref{lem:aschbacher}, it remains to  
show that $\Aut_\CF(N) \leq \Aut(\CF_N)$.  But this comes from the  
uniqueness of the saturated fusion subsystems of $p$-power index on a  
given subgroup of $S$ containing $O^p_\CF(S)$. Indeed, any morphism in  
$\alpha\in\Aut_\CF(N)$ gives a fusion preserving isomorphism from $ 
\CF_N$ to $^\alpha\CF_N$ which is another saturated fusion system on $N 
$ containing $O^p(\Aut_\CF(P))$ for any $P\le N$. By the uniqueness of  
such systems, we have $\alpha\in\Aut(\CF_N)$.
\end{proof}

Finally, we show that $O^p(\CF)=O^p(O^p(\CF))$. This is a result of Puig 
\cite[7.5]{puig:frobeniuscategories} but the proof we present here is based 
on \cite{BCGLO2007extensions}. We will need the concept of a quotient system of a fusion system.

\begin{dfn}
Let $\CF$ be a fusion system on $S$ and let $T$ be a strongly $\CF$-closed subgroup of $S$. By the {\it quotient system} $\CF/T$, we mean  
the fusion system on $S/T$, such that for any two subgroups $U$ and $V 
$ of $S$ containing $T$, $\Hom_{\CF/P}(U/P,V/P)$ is the set of  
homomorphisms induced by morphisms in $\Hom_\CF(U,V)$.
\end{dfn}

When the fusion system is saturated Puig proves in \cite{puig:frobeniussystems} and \cite[6.3]{puig:frobeniuscategories} that the saturation is inherited by the quotient system.
\begin{thm}[{\cite[6.3]{puig:frobeniuscategories}}]\label{thm:QutientFusion}
Let $\CF$ be a saturated fusion system on a finite $p$-group $S$. If $T$ is a strongly $\CF$-closed subgroup of $S$, then the quotient  
system $\CF/T$ is saturated.
\end{thm}

In fact, the above result holds even if $T$ is only weakly $\CF$-closed.

%Let $\CF$ be a saturated fusion system on a finite $p$-group $S$ and  
%let $T$ be a strongly $\CF$-closed subgroup of $S$.  If $P$ and $Q$  
%are subgroups of $S$, then any $\CF$-morphism $\varphi \colon P \to Q$  
%induces a group homomorphism $\ol{\varphi}\colon PT/T \to QT/T$ given  
%by $\ol{\varphi}(uT)=\varphi(u)T$ for $u\in P$, because $T$ is  
%strongly $\CF$-closed.  The following theorem of Puig asserts that  
%this induced map belongs to the quotient system $\CF/T$. See 
%\cite[5.10]{Cravencontrol} for an alternative proof.

\begin{thm}[{\cite[6.3]{puig:frobeniuscategories} or \cite[5.10]{Cravencontrol}}] \label{thm:CravenQuotient}
Let $\CF$ be a saturated fusion system on a finite $p$-group $S$ and  
let $T$ be a strongly $\CF$-closed subgroup of $S$.  If the $P, Q \leq S$ and $\varphi\in 
\Hom_\CF(P,Q)$, then the induced map $\ol{\varphi}\colon PT/T \to QT/T$ belongs to $\CF/T$. 
%that is, there  
%is $\widetilde{\varphi}\in\Hom_\CF(PT,QT)$ such that $\ol{\varphi}(uT)= 
%\widetilde{\varphi}(uT)$ for $u\in P$.
\end{thm}

An interesting connection between quotient systems and $O^p_\CF(S)$ is the following lemma.

\begin{lem}\label{lem:trivquotgivescontainment}
Let $\CF$ be a saturated fusion system on a finite $p$-group $S$ and let $T$ be a strongly $\CF$-closed subgroup of $S$. If $\CF/T$ is the trivial fusion system on $S/T$, then $O^p_{\CF}(S) \leq T$.
\end{lem}
\begin{proof}
If $T\le Q\le S$ and $\rho\in\Aut_\CF(Q)$ is a $p'$-automorphism, then $\rho$ induces the identity on $Q/T$,  
implying that $u^{-1}\rho(u)\in T$ for any $u\in Q$. As these generate $O^p_\CF(S)$, the result follows.
%By Theorem~ \ref{thm:CravenQuotient}, those are the generators of $\OpFS$.
\end{proof}

Getting back to the issue of proving $O^p(\CF) = O^p(O^p(\CF))$, we use the following notation $S_1:=O^p_\CF(S)$, $\CF_1:=O^p(\CF)$,  
$S_2:=O^p_{\CF_1}(S_1)$.
\[
\xymatrix{\CF \ar@{-}[r] \ar@{-}[d] & S \ar@{-}[d] \\ \CF_1 \ar@{-}[r] \ar@{-}[d] & S_1  \ar@{-}[d] \\ 
O^p(\CF_1) \ar@{-}[r] & S_2 }
\]
Using Theorem \ref{thm:p-powerindexbijection}, we need only show that $S_1 = S_2$.
%\[
%\xymatrix{&\CF \ar@{-}[r] \ar@{-}[d] & S \ar@{-}[d] & \\ O^p(\CF) = & \CF_1 \ar@{-}[r] \ar@{-}[d] & S_1 \ar@{-}[d]& = O^p_\CF(S)  \\ 
%& O^p(\CF_1) \ar@{-}[r] & S_2 & = O^p_{\CF_1}(S_1)}
%\]

\begin{prp}\label{thm:S2-str-closed}
The subgroup $S_2$ is strongly $\CF$-closed.
\end{prp}
\begin{proof}
Let $P\leq S_2$ and $\psi\in\Hom_\CF(P,S)$. By Lemma \ref{lem:aschbacher} and Proposition~ 
\ref{normal p-subsystem}, there is a decomposition $\psi=\phi \circ  
\alpha$ with $\alpha\in\Aut_\CF(S_1)$ and $\phi\in\Hom_{\CF_1} 
(\alpha(P),S_1)$. Since $\Aut_\CF(S_1) \leq \Aut(\CF_1)$, we have $ 
\alpha(P) \leq S_2$ and thus $\psi(P)=\phi(\alpha(P)) \leq S_2$  
since $S_2$ is strongly $\CF_1$-closed.
\end{proof}

By the definition of the hyperfocal subgroup, $\CF/S_1$  
and $\CF_1/S_2$ are the trivial fusion systems on $S/S_1$ and $S_1/S_2$,  
respectively.

\begin{prp}\label{thm:Op-idempotent}
With the notation as above, we have $S_2=S_1$. In particular $O^p(\CF)=O^p(O^p(\CF))$.
\end{prp}
\begin{proof}
By Lemma \ref{lem:trivquotgivescontainment}, it will suffice to show that $\CF/S_2$ is the trivial system on $S/S_2$.
By Proposition~\ref{thm:QutientFusion} and Proposition \ref{thm:S2-str-closed}, $\CF/S_2$ is a saturated fusion system on $S/S_2$. If $S_2\le Q\le S$ and $\bar\rho\in\Aut_{\CF/S_2}(Q/S_2)$ is a $p'$-automorphism, then there exist $\rho\in\Aut_\CF(Q)$ lifting $\bar\rho$  
and, raising $\rho$ to an appropriate $p$-th power, we may suppose that $\rho 
$ is also a $p'$-automorphism. Note that, in particular, $\rho$ belongs  
to $\CF_1$.  Now $\rho$ induces $p'$-automorphisms on $Q/(Q\cap S_1)$  
and on $(Q\cap S_1)/S_2$.  The induced $p'$-automorphism of $Q/(Q\cap  
S_1) \cong QS_1/S_1$ belongs to $\CF/S_1$ by Theorem~ 
\ref{thm:CravenQuotient}, and so is the identity map because $\CF/S_1$  
is the trivial fusion system on $S/S_1$.  Similarly, the induced $p'$- 
automorphism of $(Q\cap S_1)/S_2$ is the identity map.  Hence $\rho$  
itself is the identity map by \cite[5.3.2]{GorensteinFiniteGroups},  
implying the claim in step one.
\end{proof}

This proposition implies that the hyperfocal subsystem of any $p$-power index subsystem of $\CF$ is equal to the hyperfocal subsystem of  
$\CF$.

\begin{cor}\label{cor:samehyperfocal}
Let $\CF$ be a saturated fusion system on a finite $p$-group $S$. If $O^p_\CF(S)\le T\le S$ and $\CF_T$ is the unique saturated p-power index fusion subsystem of $\CF$ on $T$, then $O^p_{\CF_T}(T)=O^p_{\CF}(S)$.
\end{cor}
\begin{proof}
As $O^p(\CF)\subseteq\CF_T\subseteq\CF$ we have that $O^p_{O^p(\CF)}(O^p_ 
\CF(S))\le O^p_{\CF_T}(T)\le O^p_{\CF}(S)$. Proposition~\ref{thm:Op-idempotent} tells us that
the first and the last term in the inequality are equal and the  
corollary follows.
\end{proof}

\end{document}